\documentclass{article}
\usepackage{amsfonts}
\usepackage{amsmath}

\newtheorem{theorem}{Theorem}

\newtheorem{definition}[theorem]{Definition}

\newtheorem{proposition}[theorem]{Proposition}

\newenvironment{proof}[1][Proof]{\noindent\textbf{#1.} }{\ \rule{0.5em}{0.5em}}
\input{tcilatex}
\begin{document}

\title{Biharmonic curves in Finsler spaces}
\author{Nicoleta VOICU \\
"Transilvania" University\\
50, Iuliu Maniu str., Brasov, Romania\\
e-mail: nico.voicu@unitbv.ro}
\date{}
\maketitle

\begin{abstract}
Biharmonic curves are a generalization of geodesics, with applications in
elasticity theory and computer science. The paper proposes a first study of
biharmonic curves in spaces with Finslerian geometry, covering the following
topics: a deduction of their equations, specific properties and existence of
non-geodesic biharmonic curves for some classes of Finsler spaces.
Integration of the biharmonic equation is presented for two concrete Finsler
metrics.
\end{abstract}



\section{Introduction}

In elasticity theory, as well as in other domains, such as computer graphics
and image processing, a classical problem is that of finding the extremal
curves of the \textit{elastic energy (}or \textit{bending energy})\textit{\
functional}, defined as the integral of the squared curvature $k^{2}$ of a
curve with respect to its arc length $s$. Its extremal curves (called 
\textit{least-energy }or \textit{elastic }curves) are interpreted as the
"smoothest" curves joining some given points or with given boundary
conditions, \cite{Glass}, \cite{Horn}.

The Riemannian generalization of the elastic energy, called the \textit{%
bienergy,} is defined as: 
\begin{equation}
E_{2}(c)=\underset{c}{\int }\kappa ^{2}ds,  \label{bienergy_Riemannian}
\end{equation}%
where $\kappa $ is the geodesic curvature of the curve $c$. Critical points
of $E_{2}$, called \textit{biharmonic curves, }are described by the equation:%
\begin{equation}
\nabla _{\dot{c}}^{3}\dot{c}=R(\dot{c},\nabla _{\dot{c}}\dot{c})\dot{c}.
\label{biharmonic_riemann}
\end{equation}

In Riemannian spaces, biharmonic curves and, more generally, biharmonic
maps, were examined quite in detail, \cite{Caddeo}, \cite{Inoguchi}, \cite%
{Jiang}, \cite{Maeta}, \cite{Montaldo}, \cite{Oniciuc}, \cite{Monterde}. As
a first remark, geodesics of $\nabla $ are always biharmonic (actually, they
are minimum points for the bienergy), but the converse is generally not true.

A non-harmonic biharmonic curve is called \textit{proper biharmonic. }One of
the most discussed related problems is the existence of proper biharmonic
curves. We only mention here several results, \cite{Jiang}, \cite{Montaldo}:

- In the Euclidean space $\mathbb{R}^{n},$ every biharmonic curve is a
segment of a straight line.

- On surfaces with nonpositive Gauss curvature, any biharmonic curve is a
geodesic.

- In Riemannian spaces with nonpositive sectional curvature, every closed
biharmonic curve is a geodesic.

Also, along any (unit speed) biharmonic curve, the geodesic curvature $%
\kappa $ is a constant.

\bigskip 

Passing to Finsler spaces, the interest in biharmonicity is motivated,
beyond purely mathematical reasons, by the potential of Finsler geometry for
models in elasticity theory, \cite{Bucataru}, \cite{Yajima1}. In the present
paper, we find the equations of biharmonic curves in Finsler spaces and
study the Finslerian counterparts of the above mentioned results. 

The paper is organized as follows. Section 2 is an introductory one, devoted
to the necessary notions and results in Finsler geometry. In Section 3, we
define the bienergy of a Finslerian curve and determine the equations of
biharmonic curves. In Section 4, we show that Finslerian biharmonic curves
have constant geodesic curvature. The next section is devoted to the problem
of the existence of proper biharmonic curves in several special cases. For
instance, we show that, to a difference from the Riemannian case, in
dimension greater or equal to three, there exist flat (locally Minkowskian)\
Finsler spaces which admit proper biharmonic curves; also, we determine
proper biharmonic curves for two concrete examples of Finsler metrics. In
Section 6, we present a brief discussion of the obtained results.

\section{Finsler structures}

In this section, we present known notions and results in Finsler geometry to
be used in the following.\ Consider a $\mathcal{C}^{\infty }$-smooth, $n$%
-dimensional differentiable manifold $M$ and denote by $%
(x,y):=(x^{i},y^{i})_{i=\overline{1,n}},$ the local coordinates on its
tangent bundle $(TM,\pi ,M).$ By $_{,i}$ and $_{\cdot i}$ we will mean
partial differentiation with respect $x^{i}$ and $y^{i}$ respectively; by $%
\Gamma (E)$ we will mean the module of the sections of some bundle $E$ over $%
M.$

A \textit{Finsler structure, \cite{Shen}, \cite{Bucataru},} on the manifold $%
M$ is a function $F:TM\rightarrow \mathbb{R}$ with the properties:

1) $F(x,y)\ $is $\mathcal{C}^{\infty }$-smooth for $y\not=0$ and continuous
at $y=0.$

2)$\ F(x,\lambda y)=\lambda F(x,y),$ $\forall \lambda >0$.

3)\ The \textit{Finslerian metric tensor}:%
\begin{equation}
g_{ij}(x,y):=\dfrac{1}{2}(F^{2}(x,y))_{\cdot ij}  \label{metric_tensor_F}
\end{equation}%
is positive definite.

By its dependence on the fiber coordinate $y,$ the Finslerian metric tensor $%
g_{ij}$ "lives"\ on the tangent bundle $TM,$ hence it is natural to perform
the study on this bundle\footnote{%
Another natural variant is the pullback bundle $\pi ^{\ast }TM,$ used in 
\cite{Mo-book}, \ \cite{Shen}.}, \cite{Bucataru}, \cite{Lagrange}.

The arc length of a curve $c$ on the Finsler space $(M,g)$ is given, \cite%
{Shen}, by:%
\begin{equation}
l(c)=~\underset{c}{\int }F(x,dx).  \label{arc_length}
\end{equation}

A key ingredient in Finsler geometry, which considerably simplifies
calculations, are adapted frames given by Ehresmann (nonlinear, \cite%
{Bucataru}, \cite{Lagrange}, \cite{Shen}) connections on $TM.$ The typical
choice is the \textit{Cartan nonlinear connection }$TTM=HTM\oplus VTM$,%
\textit{\ }having the property that complete lifts to $TM$ of tangent
vectors to unit speed geodesics of $M$ are always horizontal. The Cartan
nonlinear connection is locally defined by the coefficients%
\begin{equation}
G_{~j}^{i}:=G_{~\cdot j}^{i},  \label{Cartan_conn}
\end{equation}%
where $G^{i}=G^{i}(x,y)$ are obtained from the equations of unit speed
geodesics of $(M,F)$:

\begin{equation}
\dfrac{dy^{i}}{ds}+2G^{i}(x,y)=0,~\ y=\dot{x}.  \label{spray_coeff}
\end{equation}%
The (generally, nonholonomic)\ adapted basis on $\Gamma (TTM)$ is then given
by: 
\begin{equation}
(\delta _{i}=\dfrac{\partial }{\partial x^{i}}-G_{~i}^{j}(x,y)\dfrac{%
\partial }{\partial y^{j}},~~\ \ \dot{\partial}_{i}=\dfrac{\partial }{%
\partial y^{i}})  \label{general_adapted_basis}
\end{equation}%
and its dual is: $(dx^{i},~\ \delta y^{i}=dy^{i}+G_{~j}^{i}dx^{j}).$ Any
vector field $X$ on $TM$ can be decomposed as: $X=hX+vX,$ $\ hX=X^{i}\delta
_{i},$ $vX=\hat{X}^{i}\dot{\partial}_{i};$ its horizontal component $hX$ and
its vertical component $vX$ are vector fields on $TM$.

In the presence of a nonlinear connection, vector fields on $M$ can be
identified with horizontal vector fields on $TM$ via the horizontal lift 
\begin{equation*}
l^{h}:\Gamma (TM)\rightarrow \Gamma (HTM),~\ ~\ \ \mathbf{X}=\mathbf{X}%
^{i}\partial _{i}\mapsto l^{h}(\mathbf{X})=(\mathbf{X}^{i}\circ \pi )\delta
_{i}.
\end{equation*}%
The Finslerian metric tensor $g$ gives rise to a scalar product%
\begin{equation*}
\left\langle X,Y\right\rangle =g_{ij}X^{i}Y^{j}
\end{equation*}%
of horizontal vector fields $X=X^{i}\delta _{i},Y=Y^{j}\delta _{j}\in \Gamma
(HTM).$

The following (1,1)-type tensor fields on $TM$:%
\begin{equation}
J=\dot{\partial}_{i}\otimes dx^{i},~~\theta =\delta _{i}\otimes \delta y^{i},
\label{tangent_structure}
\end{equation}%
have the role of mapping horizontal vector fields into vertical ones and
vice-versa ($J(\delta _{i})=\dot{\partial}_{i},$ $J(\dot{\partial}_{i})=0;$ $%
\theta (\dot{\partial}_{i})=\delta _{i},$ $\theta (\delta _{i})=0$).

\bigskip

As covariant differentiation\textbf{\ }rule for tensors on $TM,$ we will use
the one given by the \textit{Cartan affine connection} $D$ on $TM,$ \cite%
{Bucataru}, \cite{Lagrange}, locally described by the coefficients:%
\begin{equation}
D_{\delta _{k}}\delta _{j}=~\Gamma _{~jk}^{i}\delta _{i},~\ D_{\delta _{k}}%
\dot{\partial}_{j}=\Gamma _{~jk}^{i}\dot{\partial}_{i},~\ D_{\dot{\partial}%
_{k}}\delta _{j}=C_{~jk}^{i}\delta _{i}~D_{\dot{\partial}_{k}}\dot{\partial}%
_{j}=C_{~jk}^{i}\dot{\partial}_{i},  \label{chern_conn}
\end{equation}%
where $\Gamma _{~jk}^{i}=\Gamma _{~jk}^{i}(x,y),$ $%
C_{~jk}^{i}=C_{~jk}^{i}(x,y)$ are as follows:%
\begin{equation}
\Gamma _{~jk}^{i}=\dfrac{1}{2}g^{ih}(\delta _{k}g_{hj}+\delta
_{j}g_{hk}-\delta _{h}g_{jk});~\ \ \ \ \ C_{~jk}^{i}=\dfrac{1}{2}%
g^{ih}g_{hj\cdot k}.  \label{Cartan_tensor_comps}
\end{equation}%
The Cartan connection is metrical and preserves by parallelism the
horizontal and vertical distributions on $TTM$, i.e.,%
\begin{equation}
D_{X}(hY)=h~D_{X}Y,~~~D_{X}(vY)=v~D_{X}Y,~\ \forall X,Y\in \Gamma (TTM).
\label{d-connection}
\end{equation}

Its torsion $\mathcal{T}$ has, in the adapted frame, the following nonzero
local components:%
\begin{equation}
\mathcal{T}(X,Y)=R_{~jk}^{i}\dot{\partial}_{i}\otimes dx^{k}\otimes
dx^{j}+C_{~jk}^{i}\delta _{i}\otimes \delta y^{k}\otimes dx^{j}+P_{~jk}^{i}%
\dot{\partial}_{i}\otimes \delta y^{k}\otimes dx^{j},  \label{torsion_def}
\end{equation}%
with $C_{~jk}^{i}$ as in (\ref{Cartan_tensor_comps}) and:%
\begin{equation}
R_{~jk}^{i}=\delta _{k}G_{~j}^{i}-\delta _{j}G_{~k}^{i},~\
P_{~jk}^{i}=y^{l}C_{~jk|l}^{i}  \label{local_expr_P}
\end{equation}%
(where $_{|l}:=D_{\delta _{l}}$). The first of the three terms in (\ref%
{torsion_def})\ is the curvature of the nonlinear connection $N$ and the
other two define two specific geometric objects to Finsler geometry:

2) the \textit{Cartan tensor\footnote{%
The Cartan tensor is a measure of the non-Riemannian character of the
Finsler metric; the equality $C=0$ is equivalent to the fact that the metric
is Riemannian.}}:\ $C(X,Y):=h\mathcal{T}(JX,Y),~\ \forall X,Y\in \Gamma
(HTM),$ in local writing:\ $C=C_{~jk}^{i}\delta _{i}\otimes dx^{k}\otimes
dx^{j}$;

3)\ the \textit{Landsberg tensor }$P(X,Y):=\theta \mathcal{T}(JX,Y),~\
\forall X,Y\in \Gamma (HTM),$ in local coordinates:\ $P:=P_{~jk}^{i}\delta
_{i}\otimes dx^{k}\otimes dx^{j}$.

The Cartan and the Landsberg tensors obey similar properties:

i) the scalar products $\left\langle C(X,Y),W\right\rangle ,$ $\left\langle
P(X,Y),W\right\rangle $ are totally symmetric in $X,Y,W\in \Gamma (HTM);$

ii) contraction of $C$ and $P$ with $y$ provides zero, i.e.:%
\begin{equation}
C(l^{h}(y),~\cdot )=0,~P(l^{h}(y),~\cdot )=0.  \label{property_C}
\end{equation}

\bigskip

The curvature $\mathcal{R}$ of the Cartan connection $D$ also splits in the
adapted basis into horizontal and vertical components. The local expressions
of these components can be found, for instance, in \cite{Bucataru}, \cite%
{Lagrange}; here, we will only mention some of their properties:

a) $\forall U,V\in \Gamma (TM):~\mathcal{R}(vU,hV)(l^{h}y)=P(\theta U,hV);$ $%
~\ \ \mathcal{R}(vU,vV)(l^{h}y)=0;$

b)\ the $h(hh)$ component of $\mathcal{R}$ gives rise to the \textit{Jacobi
endomorphism} of $(M,F)$, \cite{Bucataru}, understood here as a (1,1)-type
horizontal tensor:%
\begin{equation}
R(X)=\mathcal{R}(X,l^{h}y)l^{h}y,~\ \forall X\in \Gamma (HTM);  \label{def_R}
\end{equation}%
in local writing, $R=R_{~j}^{i}\delta _{i}\otimes dx^{j},~\
R_{~j}^{i}=R_{~kj}^{i}y^{k}$. The components $R_{ik}$ are symmetric,
therefore, 
\begin{equation}
\left\langle R(X),Y\right\rangle =\left\langle R(Y),X\right\rangle ,~\
\forall X,Y\in \Gamma (HTM).  \label{property_R}
\end{equation}

If the system $\{l^{h}(y),X\}$ is orthonormal, then $K(X):=\dfrac{1}{F^{2}}%
\left\langle R(X),X\right\rangle $ is equal to the \textit{flag curvature }%
of $(M,F),$ \cite{Shen}, which represents an analogue of sectional curvature
from Riemannian spaces.

\textbf{Particular cases:}

1) In\ Riemannian spaces, we have $g_{ij}=g_{ij}(x),$ $C=0,$ $P=0,$ while
the $h(hh)$ component $\mathcal{R}(hU,hV)hW$ ($U,V,W\in \Gamma (TM)$) is, up
to a horizontal lift, the Riemannian curvature of $(M,g).$

2) \textit{Locally Minkowski spaces} ("flat" Finsler spaces) are spaces in
which there exists around any point a local chart in which $%
g_{ij}=g_{ij}(y). $ In this case, $\Gamma _{~jk}^{i}=0,$ $R=0,$ $P=0,$ but,
generally, $C\not=0.$ Geodesics of a locally Minkowski space are straight
lines.

3) \textit{Landsberg spaces} are defined by the condition $P=0.$

\section{Bienergy of a curve and its first variation}

Consider a curve $c:I\rightarrow M,$ $s\mapsto (x^{i}(s))$ (where $I\subset 
\mathbb{R}$ is an interval), parametrized by its arc length and its lift $%
c^{\prime }:I\rightarrow TM,$ $s\mapsto (x^{i}(s),\dot{x}^{i}(s))$ to $TM;$
that is, $c$ has unit Finslerian speed:%
\begin{equation}
F(x,\dot{x})=1.  \label{unit_speed}
\end{equation}%
Tangent vectors of $c^{\prime }$ are expressed in the adapted basis as:%
\begin{equation*}
T^{\prime }:=c_{\ast }^{\prime }(\partial _{s})=y^{i}\delta _{i}+\dfrac{%
\delta y^{i}}{ds}\dot{\partial}_{i},~\ \ y^{i}=\dot{x}^{i}.
\end{equation*}%
We will denote simply by $D:=D^{c^{\prime }}$ the pullback of the Cartan
affine connection $D$\textit{\ }by $c^{\prime }.$ Taking into account that $%
\Gamma _{~jk}^{i}y^{k}=G_{~j}^{i}$ (\cite{Bucataru}, \cite{Lagrange}), the
components of $T^{\prime }$ in the adapted basis can be re-expressed%
\footnote{%
Here and in the following, covariant differentiation is made, \cite{Shen},
with "reference vector" $T,$ i.e., in the expressions of the connection
coefficients $\Gamma =\Gamma (x,y)$ and $C=C(x,y),$ we set $y^{i}=\dot{x}^{i}
$).} as: 
\begin{equation}
T:=hT^{\prime }=l^{h}(y),~\ \ vT^{\prime }=J(D_{\partial _{s}}T)
\label{horizontal_T}
\end{equation}

\bigskip 

The \textit{energy} of the curve $c$ is given, \cite{Mo-book}, by%
\begin{equation}
E_{1}(c)=\dfrac{1}{2}\underset{c}{\int }\left\langle T,T\right\rangle ds.
\label{energy}
\end{equation}%
Critical points of the energy, i.e., \textit{harmonic curves}, \cite{Mo-book}%
, are constant speed geodesics of $(M,F),$ described by:%
\begin{equation}
D_{\partial _{s}}T=0.  \label{harmonic_curves}
\end{equation}%
The horizontal vector field%
\begin{equation}
\tau (c):=D_{\partial _{s}}T,  \label{tension}
\end{equation}%
defined along the curve $c,$ is called (by analogy with the Riemannian
case), \cite{Mo-book}, the \textit{tension }of $c$. The norm of the tension
field, i.e., $\kappa =\sqrt{\left\langle D_{\partial _{s}}T,D_{\partial
_{s}}T\right\rangle },$ gives the \textit{geodesic curvature} of the curve $%
c;$ we will speak more in detail about geodesic curvature in the next
section. 

It appears as natural to define the bienergy as follows:

\begin{definition}
1)\ The \textit{bienergy }of a curve $c$ on a Finsler space $(M,F)$ is:%
\begin{equation}
E_{2}(c)=\dfrac{1}{2}\underset{c}{\int }\left\langle D_{\partial
_{s}}T,D_{\partial _{s}}T\right\rangle ds.  \label{bienergy}
\end{equation}%
2)\ Critical points of the bienergy are called biharmonic curves of $(M,F).$
\end{definition}

\bigskip

There holds the following result:

\begin{theorem}
A curve $c$ on the Finsler space $(M,F)$ is biharmonic if and only if: 
\begin{equation}
\tau _{2}(c):=D_{\partial _{s}}^{3}T+R(D_{\partial _{s}}T)-P(D_{\partial
_{s}}T,D_{\partial _{s}}T)+C(D_{\partial _{s}}T,D_{\partial _{s}}^{2}T)=0.
\label{bitension}
\end{equation}
\end{theorem}

\begin{proof}
Let us consider smooth variations $c_{\varepsilon }(s)=\tilde{c}(\varepsilon
,s),$ $c_{0}(s)=c(s)$ with fixed endpoints of $c$ and their lifts $%
c_{\varepsilon }^{\prime }=(c_{\varepsilon },\dfrac{dc_{\varepsilon }}{ds})$
to $TM.$ With the notations: 
\begin{equation}
\mathbf{T}^{\prime }:=\tilde{c}_{\ast }^{\prime }(\partial _{s}),\ \mathbf{V}%
^{\prime }:=\tilde{c}_{\ast }^{\prime }(\partial _{\varepsilon });~\ \ \ 
\mathbf{T}:=h\mathbf{T}^{\prime },~\mathbf{V}:=h\mathbf{V}^{\prime },
\label{boldface_T}
\end{equation}%
we have $v\mathbf{T}^{\prime }=J(D_{\partial _{s}}\mathbf{T})$ and, at $%
\varepsilon =0$: 
\begin{equation*}
T=\mathbf{T}\rfloor _{\varepsilon =0},~\ V:=\mathbf{V}\rfloor _{\varepsilon
=0}.
\end{equation*}%
Taking into account that the Cartan connection $D$ is metrical, we can write:%
\begin{equation}
\dfrac{dE_{2}}{d\varepsilon }(c_{\varepsilon })=\dfrac{1}{2}\underset{c}{%
\int }\dfrac{d}{d\varepsilon }\left\langle D_{\partial _{s}}\mathbf{T}%
,D_{\partial _{s}}\mathbf{T}\right\rangle ds=~\underset{c}{\int }%
\left\langle D_{\partial _{\varepsilon }}D_{\partial _{s}}\mathbf{T}%
,D_{\partial _{s}}\mathbf{T}\right\rangle ds.  \label{first_var1}
\end{equation}%
Further, since $D$ preserves the distributions generated by $N,$ (\ref%
{boldface_T}),%
\begin{equation*}
\begin{array}{c}
D_{\partial _{\varepsilon }}D_{\partial _{s}}\mathbf{T}=\mathcal{R}(\mathbf{V%
}^{\prime },\mathbf{T}^{\prime })\mathbf{T}+~D_{\partial _{s}}D_{\partial
_{\varepsilon }}(h\mathbf{T}^{\prime })~=\mathcal{R}(\mathbf{V}^{\prime },%
\mathbf{T}^{\prime })\mathbf{T}+ \\ 
+~D_{\partial _{s}}(hD_{\partial _{\varepsilon }}\mathbf{T}^{\prime })~=%
\mathcal{R}(\mathbf{V}^{\prime },\mathbf{T}^{\prime })\mathbf{T}%
+~D_{\partial _{s}}(h\mathcal{T}(\mathbf{V}^{\prime },\mathbf{T}^{\prime
}))+D_{\partial _{s}}^{2}(h\mathbf{V}^{\prime });%
\end{array}%
\end{equation*}%
with $\mathbf{V}=h\mathbf{V}^{\prime }$ and substituting into (\ref%
{first_var1}), we get:%
\begin{equation}
\dfrac{dE_{2}}{d\varepsilon }=\underset{c}{\int }\left\langle \mathcal{R}(%
\mathbf{V}^{\prime },\mathbf{T}^{\prime })\mathbf{T}+~D_{\partial _{s}}(h%
\mathcal{T}(\mathbf{V}^{\prime },\mathbf{T}^{\prime }))+D_{\partial _{s}}^{2}%
\mathbf{V},~D_{\partial _{s}}\mathbf{T}\right\rangle ds.  \label{firstvar2}
\end{equation}%
Let us evaluate each term in the above integral. First of all, the curvature
term is:%
\begin{eqnarray*}
&&I_{1}:=\underset{c}{\int }\left\langle \mathcal{R}(\mathbf{V}^{\prime },%
\mathbf{T}^{\prime })\mathbf{T},~D_{\partial _{s}}\mathbf{T}\right\rangle ds=%
\underset{c}{\int }\left\langle \mathcal{R}(\mathbf{V},\mathbf{T})\mathbf{T}+%
\mathcal{R}(v\mathbf{V}^{\prime },\mathbf{T})\mathbf{T}\right. + \\
&&+\left. \mathcal{R}(\mathbf{V},v\mathbf{T}^{\prime })\mathbf{T}+\mathcal{R}%
(v\mathbf{V}^{\prime },v\mathbf{T}^{\prime })\mathbf{T},~D_{\partial _{s}}%
\mathbf{T}\right\rangle ds.
\end{eqnarray*}%
Along each of the curves $c_{\varepsilon },$ we have $\mathbf{T}~=l^{h}(y)$;
taking into account the properties of the curvature components, the terms $%
\left\langle \mathcal{R}(v\mathbf{V}^{\prime },\mathbf{T})\mathbf{T}%
,D_{\partial _{s}}\mathbf{T}\right\rangle $ and $\left\langle \mathcal{R}(v%
\mathbf{V}^{\prime },v\mathbf{T}^{\prime })\mathbf{T},~D_{\partial _{s}}%
\mathbf{T}\right\rangle $ vanish, while the two remaining ones can be
written in a simpler form:%
\begin{eqnarray*}
\left\langle \mathcal{R}(\mathbf{V},\mathbf{T})\mathbf{T},D_{\partial _{s}}%
\mathbf{T}\right\rangle ~ &=&\left\langle R(\mathbf{V}),D_{\partial _{s}}%
\mathbf{T}\right\rangle \overset{(\ref{property_R})}{=}\left\langle
R(D_{\partial _{s}}\mathbf{T}),\mathbf{V}\right\rangle ; \\
\left\langle \mathcal{R}(h\mathbf{V}^{\prime },v\mathbf{T}^{\prime })\mathbf{%
T},D_{\partial _{s}}\mathbf{T}\right\rangle &=&-\left\langle P(D_{\partial
_{s}}\mathbf{T,V}),D_{\partial _{s}}\mathbf{T})\right\rangle =-\left\langle
P(D_{\partial _{s}}\mathbf{T},D_{\partial _{s}}\mathbf{T}),\mathbf{V}%
)\right\rangle ;
\end{eqnarray*}%
therefore, we can write:%
\begin{equation}
I_{1}=\underset{c}{\int }\left\langle R(D_{\partial _{s}}\mathbf{T}%
)-P(D_{\partial _{s}}\mathbf{T},D_{\partial _{s}}\mathbf{T}),\mathbf{V}%
\right\rangle ds.  \label{I1}
\end{equation}%
The second (torsion) term in (\ref{firstvar2}), i.e.:%
\begin{equation*}
I_{2}:=\underset{c}{\int }\left\langle D_{\partial _{s}}(h\mathcal{T}(%
\mathbf{V}^{\prime },\mathbf{T}^{\prime })),~D_{\partial _{s}}\mathbf{T}%
\right\rangle ds
\end{equation*}%
will be integrated by parts:%
\begin{equation*}
I_{2}=\underset{c}{\int }\{\partial _{s}\left\langle h\mathcal{T}(\mathbf{V}%
^{\prime },\mathbf{T}^{\prime }),D_{\partial _{s}}\mathbf{T}\right\rangle
-\left\langle (h\mathcal{T}(\mathbf{V}^{\prime },\mathbf{T}^{\prime
})),~D_{\partial _{s}}^{2}\mathbf{T}\right\rangle \}ds.
\end{equation*}%
Since the variation has fixed endpoints, we have: $\underset{c}{\int }%
\partial _{s}\left\langle h\mathcal{T}(\mathbf{V}^{\prime },\mathbf{T}%
^{\prime }),D_{\partial _{s}}\mathbf{T}\right\rangle ds=0$; using (\ref%
{torsion_def}) and (\ref{property_C}) in the evaluation of the remaining
term, we are led to:%
\begin{equation*}
I_{2}=\underset{c}{\int }\left\langle C(D_{\partial _{s}}\mathbf{T},\mathbf{V%
}),D_{\partial _{s}}^{2}\mathbf{T}\right\rangle ds=\underset{c}{\int }%
\left\langle C(D_{\partial _{s}}\mathbf{T},D_{\partial _{s}}^{2}\mathbf{T}),~%
\mathbf{V}\right\rangle ds.
\end{equation*}%
Finally, integrating twice by parts the third term in (\ref{firstvar2}),
i.e.,%
\begin{equation*}
I_{3}:=\underset{c}{\int }\left\langle D_{\partial _{s}}^{2}\mathbf{V}%
,~D_{\partial _{s}}\mathbf{T}\right\rangle ds,
\end{equation*}%
we are led to:%
\begin{equation*}
I_{3}:=\underset{c}{\int }\left\langle D_{\partial _{s}}^{3}\mathbf{T},%
\mathbf{V}\right\rangle ds.
\end{equation*}%
Summing up and evaluating at $\varepsilon =0,$ the first variation of the
bienergy is:%
\begin{eqnarray*}
&&\dfrac{dE_{2}}{d\varepsilon }=I_{1}+I_{2}+I_{3}= \\
&=&\underset{c}{\int }\left\langle D_{\partial _{s}}^{3}T+R(D_{\partial
_{s}}T)-P(D_{\partial _{s}}T,D_{\partial _{s}}T)+C(D_{\partial
_{s}}T,D_{\partial _{s}}^{2}T),V\right\rangle ds,
\end{eqnarray*}%
which proves the statement.
\end{proof}

The vector field $\tau _{2}(c)$ in (\ref{bitension}) is called the \textit{%
bitension }of $c.$ In the following, we will also write the biharmonic
equation (\ref{bitension}) as:%
\begin{equation}
D_{\partial _{s}}^{3}T=\mathcal{A},  \label{bitension_A}
\end{equation}%
where 
\begin{equation}
\mathcal{A}:=-R(D_{\partial _{s}}T)+P(D_{\partial _{s}}T,D_{\partial
_{s}}T)-C(D_{\partial _{s}}T,D_{\partial _{s}}^{2}T).  \label{def_A}
\end{equation}

\bigskip

\textbf{Remark. }Any unit speed geodesic on $M$ is biharmonic; a
non-geodesic biharmonic curve will be called \textit{proper biharmonic.}

\bigskip

\textbf{Particular cases:}

1)\ If $M$ is a Riemannian space, then the equation of biharmonic curves
reduces to:%
\begin{equation*}
D_{\partial _{s}}^{3}T+R(D_{\partial _{s}}T)=0\Leftrightarrow D_{\partial
_{s}}^{3}T-\mathcal{R}(T,D_{\partial _{s}}T)T=0.
\end{equation*}%
Up to a projection onto the base manifold, this is the known equation (\ref%
{biharmonic_riemann}).

2)\ In locally Minkowski spaces, equation (\ref{bitension}) becomes:%
\begin{equation*}
D_{\partial _{s}}^{3}T+C(D_{\partial _{s}}T,D_{\partial _{s}}^{2}T)=0.
\end{equation*}

3)\ In Landsberg spaces, $\tau _{2}(c)=D_{\partial _{s}}^{3}T+R(D_{\partial
_{s}}T)+C(D_{\partial _{s}}T,D_{\partial _{s}}^{2}T).$

\section{Frenet frame and geodesic curvature}

We will suppose that, along $c,$ the vectors $\{T,D_{\partial
_{s}}T,D_{\partial _{s}}^{2}T,...,D_{\partial _{s}}^{n}T\}$ are linearly
independent; in particular, this implies that $D_{\partial _{s}}T\not=0,$
i.e., $c$ is non-geodesic. Similarly to Riemannian geometry, \cite{Caddeo}, 
\cite{Maeta}, we get the \textit{Frenet frame }along $c$ as the
orthonormalization of these vectors, obeying the relations:%
\begin{equation}
\left\{ 
\begin{array}{l}
e_{1}=T=l^{h}(y) \\ 
D_{\partial _{s}}e_{1}=\kappa _{1}e_{2} \\ 
D_{\partial _{s}}e_{2}=-\kappa _{1}e_{1}+\kappa _{2}e_{3} \\ 
D_{\partial _{s}}e_{i}=-\kappa _{i-1}e_{i-1}+\kappa _{i}e_{i+1},~\ i=2,..,n-1
\\ 
... \\ 
D_{\partial _{s}}e_{n}=-\kappa _{n-1}e_{n-1};%
\end{array}%
\right.   \label{Frenet_frame}
\end{equation}%
here, $e_{1},...,e_{n}$ are regarded as sections of the bundle $(c^{\prime
})^{-1}(HTM)$); the first curvature\footnote{%
In dimension 2, the first curvature is replaced by the signed curvature, 
\cite{Montaldo}.} $\kappa _{1}$, i.e., 
\begin{equation}
\kappa _{1}:=\sqrt{\left\langle D_{\partial _{s}}T,D_{\partial
_{s}}T\right\rangle }=\kappa   \label{k1}
\end{equation}%
is the geodesic curvature of $c.$

Using (\ref{Frenet_frame}), the expressions of the covariant derivatives $%
D_{\partial _{s}}^{i}T,$ $i=\overline{1,3}$ in the Frenet frame are:%
\begin{equation}
\left\{ 
\begin{array}{l}
D_{\partial _{s}}T=\kappa _{1}e_{2} \\ 
D_{\partial _{s}}^{2}T=-\kappa _{1}^{2}e_{1}+\kappa _{1}^{\prime
}e_{2}+\kappa _{1}\kappa _{2}e_{3} \\ 
D_{\partial _{s}}^{3}T=-3\kappa _{1}\kappa _{1}^{\prime }e_{1}+(\kappa
_{1}^{\prime \prime }-\kappa _{1}^{3}-\kappa _{1}\kappa
_{2}^{2})e_{2}+(2\kappa _{1}^{\prime }\kappa _{2}+\kappa _{1}\kappa
_{2}^{\prime })e_{3}+\kappa _{1}\kappa _{2}\kappa _{3}e_{4}.%
\end{array}%
\right.  \label{Frenet_derivatives_T}
\end{equation}

Substituting (\ref{Frenet_derivatives_T}) into the biharmonic equation (\ref%
{bitension_A}) and using the fact that the Cartan tensor $C$ vanishes on $%
e_{1},$ we have:%
\begin{equation}
\mathcal{A}=-\kappa _{1}R(e_{2})+\kappa _{1}^{2}P(e_{2},e_{2})-\kappa
_{1}\kappa _{1}^{\prime }C(e_{2},e_{2})-\kappa _{1}^{2}\kappa
_{2}C(e_{2},e_{3}).  \label{definition_A}
\end{equation}%
Taking the scalar product of the third equality (\ref{Frenet_derivatives_T}%
)\ with $e_{i},$ $i=\overline{1,n}$ and noticing that $\left\langle \mathcal{%
A},e_{1}\right\rangle =0,$ we find:

\begin{proposition}
A unit speed curve $c:I\rightarrow M$ is biharmonic if and only if:%
\begin{equation}
\left\{ 
\begin{array}{l}
\kappa _{1}\kappa _{1}^{\prime }=0; \\ 
\kappa _{1}^{\prime \prime }-\kappa _{1}^{3}-\kappa _{1}\kappa
_{2}^{2}=\left\langle \mathcal{A},e_{2}\right\rangle \\ 
2\kappa _{1}^{\prime }\kappa _{2}+\kappa _{1}\kappa _{2}^{\prime
}=\left\langle \mathcal{A},e_{3}\right\rangle \\ 
\kappa _{1}\kappa _{2}\kappa _{3}=\left\langle \mathcal{A}%
,e_{4}\right\rangle ; \\ 
0=\left\langle \mathcal{A},e_{i}\right\rangle ,~\ i\geq 5.%
\end{array}%
\right.  \label{biharmonic_eqn_frenet}
\end{equation}
\end{proposition}

The assumption that $c$ is non-geodesic means that $\kappa _{1}\not=0,$
which, by (\ref{biharmonic_eqn_frenet}), gives us $\kappa _{1}^{\prime }=0.$
That is, if $c$ is a proper biharmonic curve parametrized by arc length then:%
\begin{equation*}
\kappa _{1}=const.\not=0.
\end{equation*}%
Since in the case when $\kappa _{1}=0$ (i.e., $c$ is a geodesic), the curve
is trivially biharmonic, we get a similar result to the one in Riemannian
geometry:

\begin{theorem}
\label{geodesic_curvature}Along any biharmonic curve of a Finsler space, the
geodesic curvature $\kappa _{1}=\left\Vert D_{\partial _{s}}T\right\Vert $
is constant.
\end{theorem}

\section{Existence of proper biharmonic curves}

\subsection{Closed biharmonic curves}

We will assume in the following that $c$ is a smooth closed biharmonic
curve, i.e., $c$ and its derivatives coincide at its endpoints. We can
therefore write: 
\begin{equation*}
0=\underset{c}{\int }\dfrac{1}{2}\dfrac{d^{2}}{ds^{2}}\left\langle
D_{\partial _{s}}T,D_{\partial _{s}}T\right\rangle ds=\underset{c}{\int }%
\{\left\langle D_{\partial _{s}}^{3}T,D_{\partial _{s}}T\right\rangle
+\left\langle D_{\partial _{s}}^{2}T,D_{\partial _{s}}^{2}T\right\rangle
\}ds.
\end{equation*}

Using the biharmonic equation, this is: 
\begin{equation}
0=~\underset{c}{\int }\{\left\langle \mathcal{A},D_{\partial
_{s}}T\right\rangle +\left\langle D_{\partial _{s}}^{2}T,D_{\partial
_{s}}^{2}T\right\rangle \}ds.  \label{aux1}
\end{equation}

Let us evaluate $\underset{c}{\int }\{\left\langle \mathcal{A},D_{\partial
_{s}}T\right\rangle ds.$ By virtue of the total symmetry of $C,$ the term
involving it in the mentioned integral is: 
\begin{equation}
\underset{c}{\int }\left\langle -C(D_{\partial _{s}}T,D_{\partial
_{s}}^{2}T),D_{\partial _{s}}T\right\rangle ds=-\underset{c}{\int }%
\left\langle C(D_{\partial _{s}}T,D_{\partial _{s}}T),D_{\partial
_{s}}^{2}T\right\rangle ds.  \label{aux2}
\end{equation}%
Since the curve is closed, we have $\underset{c}{\int }\partial
_{s}\left\langle (C(D_{\partial _{s}}T,D_{\partial _{s}}T)),D_{\partial
_{s}}T\right\rangle =0$, hence, integration by parts in (\ref{aux2}) leads
to:%
\begin{equation*}
\begin{array}{c}
\underset{c}{\int }\left\langle -C(D_{\partial _{s}}T,D_{\partial
_{s}}T),D_{\partial _{s}}^{2}T\right\rangle ds=\underset{c}{\int }%
\left\langle D_{\partial _{s}}(C(D_{\partial _{s}}T,D_{\partial
_{s}}T)),D_{\partial _{s}}T\right\rangle ds= \\ 
=\underset{c}{\int }\left\langle (D_{\partial _{s}}C)(D_{\partial
_{s}}T,D_{\partial _{s}}T),D_{\partial _{s}}T\right\rangle +2\left\langle
C(D_{\partial _{s}}T,D_{\partial _{s}}^{2}T),D_{\partial _{s}}T\right\rangle
\}ds%
\end{array}%
\end{equation*}

Consider the following tensor acting on horizontal vector fields:%
\begin{equation}
\tilde{C}(X,Y,Z):=(D_{JX}C)(Y,Z),~\ \forall X,Y,Z\in \Gamma (HTM);\ 
\label{definition_C'}
\end{equation}%
then, using the total symmetry of $C$ and (\ref{local_expr_P}), we can write:%
\begin{equation*}
\begin{array}{l}
\underset{c}{\int }\left\langle -C(D_{\partial _{s}}T,D_{\partial
_{s}}^{2}T),D_{\partial _{s}}T\right\rangle ds=\dfrac{1}{3}\underset{c}{\int 
}\left\langle (D_{\partial _{s}}C)(D_{\partial _{s}}T,D_{\partial
_{s}}T),D_{\partial _{s}}T\right\rangle ds= \\ 
=\dfrac{1}{3}\underset{c}{\int }\{\left\langle P(D_{\partial
_{s}}T,D_{\partial _{s}}T)+\tilde{C}(D_{\partial _{s}}T,D_{\partial
_{s}}T,D_{\partial _{s}}T),D_{\partial _{s}}T\right\rangle \}ds.%
\end{array}%
\end{equation*}%
Substituting into (\ref{aux1}), we finally have:%
\begin{equation}
0=~\underset{c}{\int }\{\mathcal{F}(D_{\partial _{s}}T)+\left\langle
D_{\partial _{s}}^{2}T,D_{\partial _{s}}^{2}T\right\rangle \}ds,
\label{aux_F}
\end{equation}%
where:%
\begin{equation}
\mathcal{F}(X):=\left\langle -R(X)+\dfrac{4}{3}P(X,X)+\dfrac{1}{3}\tilde{C}%
(X,X,X),X\right\rangle ,~\ X\in \Gamma (HTM).  \label{definition_F}
\end{equation}

\bigskip

Assume that $\mathcal{F}(X)$ is nonnegative for any $X.$ Then from (\ref%
{aux_F}), it follows that $D_{\partial _{s}}^{2}T=0.$ We will show that the
latter equality implies that $c$ is harmonic. Indeed, taking the integral:%
\begin{equation*}
0=~\underset{c}{\int }\dfrac{d}{ds}\left\langle T,D_{\partial
_{s}}T\right\rangle ds=\underset{c}{\int }\{\left\langle D_{\partial
_{s}}T,D_{\partial _{s}}T\right\rangle +\left\langle T,D_{\partial
_{s}}^{2}T\right\rangle \}ds,
\end{equation*}%
\ the last term in the right hand side is zero, hence $\left\langle
D_{\partial _{s}}T,D_{\partial _{s}}T\right\rangle =0.$ By virtue of the
positive definiteness of $g,$ this leads to:%
\begin{equation*}
D_{\partial _{s}}T=0,
\end{equation*}%
i.e., $c$ is harmonic. We have thus proved:

\begin{proposition}
If, in the Finsler space $(M,F),$ the operator 
\begin{equation*}
\mathcal{F}:\Gamma (HTM)\rightarrow \mathbb{R},~\ \ \ \mathcal{F}%
(X)=\left\langle -R(X)+\dfrac{4}{3}P(X,X)+\dfrac{1}{3}\tilde{C}%
(X,X,X),X\right\rangle
\end{equation*}%
takes only nonnegative values, then any closed biharmonic curve on $M$ is a
geodesic.
\end{proposition}

\bigskip

\textbf{Particular case: }If $g$ is Riemannian, then%
\begin{equation*}
\mathcal{F}(X)=\left\langle -R(X),X\right\rangle =\left\langle -\mathcal{R}%
(X,l^{h}y)l^{h}y,X\right\rangle =-Riem_{g}(y,\pi _{\ast }X)\cdot \sigma ,
\end{equation*}%
where $Riem_{g}$ is the sectional curvature of $(M,g)$ and $\sigma
:=g(y,y)g(\pi _{\ast }X,\pi _{\ast }X)-(g(y,\pi _{\ast }X))^{2}\geq 0.$
Thus, we obtain in this case a known result, \cite{Jiang}: any closed
biharmonic curve on a Riemannian manifold with nonpositive sectional
curvature, is harmonic.

\subsection{2-dimensional Finsler spaces}

If $(M,F)\ $an oriented Finsler surface, the Frenet equations (\ref%
{Frenet_frame}) reduce to:%
\begin{equation*}
e_{1}=T,~\ D_{\partial _{s}}e_{1}=\kappa _{1}e_{2},~\ D_{\partial
_{s}}e_{2}=-\kappa _{1}e_{1}.
\end{equation*}%
and the equations of proper biharmonic curves become:%
\begin{equation*}
\left\{ 
\begin{array}{l}
\kappa _{1}\kappa _{1}^{\prime }=0; \\ 
\kappa _{1}^{\prime \prime }-\kappa _{1}^{3}=\left\langle -\kappa
_{1}R(e_{2})+\kappa _{1}^{2}P(e_{2},e_{2})-\kappa _{1}\kappa _{1}^{\prime
}C(e_{2},e_{2}),e_{2}\right\rangle ,%
\end{array}%
\right.
\end{equation*}%
that is, $\kappa _{1}\not=0$ is a constant and:%
\begin{equation}
\kappa _{1}^{2}=\left\langle R(e_{2})-\kappa
_{1}P(e_{2},e_{2}),e_{2}\right\rangle .  \label{biharmonic_curves_2_dim}
\end{equation}

Since the left hand side of the above equality is always positive, we have a
series of immediate consequences.

1)\ In \textit{Landsberg spaces}, (\ref{biharmonic_curves_2_dim})\ becomes
similar to the corresponding relation in Riemannian spaces, \cite{Montaldo}:%
\begin{equation*}
\kappa _{1}^{2}=\left\langle R(e_{2}),e_{2}\right\rangle =K(e_{2}).
\end{equation*}

It follows:

\begin{proposition}
In 2-dimensional Landsberg spaces with nonpositive (not necessarily
constant) flag curvature, any biharmonic curve is a geodesic.
\end{proposition}

2)\ In \textit{locally Minkowski spaces}, where $R=0,$ $P=0,$ (\ref%
{biharmonic_curves_2_dim})\ becomes:\ $\kappa _{1}=0.$ In other words:

\begin{proposition}
In 2-dimensional locally Minkowski spaces, any biharmonic curve is a
straight line.
\end{proposition}

\subsection{Locally Minkowski spaces}

Assume now that $(M,F)$ is a locally Minkowski space with $\dim M\geq 3$.
Then, the biharmonic equation (\ref{bitension}) reduces to:%
\begin{equation}
D_{\partial _{s}}^{3}T+C(D_{\partial _{s}}T,D_{\partial _{s}}^{2}T)=0.
\label{biharmonic_Mink}
\end{equation}

Let us choose a local chart in which $g_{ij}=g_{ij}(y)$ and denote%
\begin{equation}
D_{\partial _{s}}T=:\tau ^{i}\delta _{i},  \label{tension_local}
\end{equation}%
the components of the tension of $c$ in the adapted basis; with this
notation, (\ref{biharmonic_Mink}) is read as: 
\begin{equation*}
\dfrac{d}{ds}\dfrac{D\tau ^{i}}{ds}+2C_{~hk}^{i}\tau ^{h}\dfrac{D\tau ^{k}}{%
ds}=0.
\end{equation*}%
Lowering indices by $g_{ij}$ and taking into account that $g_{ij\cdot
k}=2C_{ijk},$ $\dfrac{dg_{ij}}{ds}=2C_{ijk}\tau ^{k},$ this is:%
\begin{equation*}
\dfrac{d}{ds}(g_{ij}\dfrac{D\tau ^{i}}{ds})=\dfrac{D\tau _{h}}{ds}=0.
\end{equation*}%
Thus, we get the first integrals:%
\begin{equation}
\dfrac{D\tau _{h}}{ds}=\lambda _{h},~\ \ \ \lambda _{h}=const.
\label{first_integral}
\end{equation}

So far, we have seen that two classes of locally Minkowski spaces, namely:

- Euclidean spaces (\cite{Montaldo});

- 2-dimensional spaces,

do not to admit proper biharmonic curves. Using (\ref{first_integral}), we
will prove that such a result cannot be extended to arbitrary locally
Minkowskian Finsler spaces.

\begin{proposition}
There exist locally Minkowski spaces of dimension at least three, which
admit proper biharmonic curves.
\end{proposition}

\begin{proof}
The space $\mathbb{R}^{3}$ equipped with the Randers-type Minkowski norm, 
\cite{Shen}, 
\begin{equation}
F(y):=\sqrt{(y^{1})^{2}+(y^{2})^{2}+(y^{3})^{2}}+by^{3}  \label{Mink_3d}
\end{equation}%
(with $b\in (0,1)$) admits proper biharmonic curves. The detailed
calculation is presented below, in Section 5.4.2.
\end{proof}

\subsection{Examples}

\subsubsection{A 2-dimensional Randers-Numata metric}

Let us consider the unit disk $M=\{x\in \mathbb{R}^{2}~|~\delta
_{ij}x^{i}x^{j}<1\}$, equipped with the Randers-type metric:%
\begin{equation}
F(x,y)=\sqrt{\delta _{ij}y^{i}y^{j}}+\delta _{ij}x^{i}y^{j}=:\alpha +\delta
_{ij}x^{i}y^{j};  \label{Numata}
\end{equation}%
and look for its proper biharmonic curves. For this metric, we know (\cite%
{Shen} p. 70) that: 
\begin{equation}
G^{i}=\dfrac{\alpha ^{2}}{2F}y^{i}  \label{Numata_G}
\end{equation}%
(note:\ our $G^{i}$ is as in \cite{Bucataru}, \cite{Lagrange}, i.e., half of
the one in \cite{Shen}) and its flag curvature is:%
\begin{equation}
K=\dfrac{3}{4F^{4}}(\delta _{ij}y^{i}y^{j})^{2}=\dfrac{3}{4F^{4}}\alpha ^{4}
\label{Numata_R}
\end{equation}%
We notice that the space is projectively flat, i.e., its geodesics are
straight lines.

The following equalities on $(M,F)$ can be checked by direct computation:%
\begin{eqnarray}
&&g_{ij}=FF_{\cdot ij}+F_{\cdot i}F_{\cdot j}=F\alpha _{\cdot ij}+F_{\cdot
i}F_{\cdot j}  \label{Numata0} \\
&&F_{\cdot ij}=\alpha _{\cdot ij},~~F_{\cdot ijk}=\alpha _{\cdot ijk};
\label{Numata1} \\
&&C_{ijk}=\dfrac{1}{4}F_{\cdot ijk}^{2}=\dfrac{1}{2}(F_{\cdot i}F_{\cdot
jk}+F_{\cdot j}F_{\cdot ik}+F_{\cdot k}F_{\cdot ij}+FF_{\cdot ijk});
\label{Numata2} \\
&&~(\alpha ^{2})_{\cdot ijk}=0\Rightarrow \alpha \alpha _{\cdot
ijk}=-(\alpha _{\cdot i}\alpha _{\cdot jk}+\alpha _{\cdot j}\alpha _{\cdot
ik}+\alpha _{\cdot k}\alpha _{\cdot ij}).  \label{Numata3}
\end{eqnarray}

Along a unit speed curve $c:I\rightarrow M,$ $s\mapsto (x^{i}(s))$ (more
precisely, along its lift $s\mapsto c^{\prime }(s):=(x^{i}(s),y^{i}(s):=\dot{%
x}^{i}(s))$ to $TM$) there hold the relations:%
\begin{eqnarray}
F(x,y) &=&\left\Vert T\right\Vert =1,\ ~\ \ F_{\cdot i}=\dfrac{g_{ij}y^{j}}{F%
};~\ \ ~\ F_{\cdot i}\tau ^{i}=\dfrac{1}{F}\left\langle T,D_{\partial
_{s}}T\right\rangle =0;  \label{Numata4} \\
\alpha _{\cdot i}y^{i} &=&\alpha ,~~\ \ \ \alpha _{\cdot ij}y^{j}=0~~\ \ \ \
\alpha _{\cdot i}\tau ^{i}=\alpha ^{\prime }+\dfrac{\alpha ^{3}}{F},
\label{Numata5}
\end{eqnarray}%
where $\alpha ^{\prime }:=d\alpha /ds=y^{i}\delta _{i}\alpha +\tau
^{i}\alpha _{\cdot i}$ and $\tau ^{i}$ are as in (\ref{tension_local}). The
tension can be expressed as:%
\begin{equation}
\tau ^{i}=\dfrac{dy^{i}}{ds}+\Gamma _{~jk}^{i}y^{j}y^{k}=\dfrac{dy^{i}}{ds}%
+G_{~j}^{i}y^{j}=\dfrac{dy^{i}}{ds}+2G^{i}=\dfrac{dy^{i}}{ds}+\dfrac{\alpha
^{2}}{F}y^{i}.  \label{Numata_tension}
\end{equation}

\bigskip

In the following, we will write the equations (\ref{biharmonic_curves_2_dim}%
) of proper biharmonic curves as:%
\begin{equation}
\left\{ 
\begin{array}{l}
g_{ij}\tau ^{i}\tau ^{j}=\kappa _{1}^{2}(=const.\not=0); \\ 
-\kappa _{1}^{4}=\left\langle -R(D_{\partial _{s}}T),D_{\partial
_{s}}T\right\rangle +\left\langle P(D_{\partial _{s}}T,D_{\partial
_{s}}T),D_{\partial _{s}}T\right\rangle .%
\end{array}%
\right.  \label{2-dim_biharmonic}
\end{equation}

Let us examine the first equation (\ref{2-dim_biharmonic}). Using (\ref%
{Numata0}), (\ref{Numata4}), (\ref{Numata_tension}) and (\ref{Numata5}),
this condition becomes:%
\begin{equation}
\kappa _{1}^{2}=F\alpha _{\cdot ij}\tau ^{i}\tau ^{j}=F\alpha _{\cdot ij}%
\dfrac{dy^{i}}{ds}\dfrac{dy^{j}}{ds}.  \label{Numata_k}
\end{equation}%
We have: $\alpha _{\cdot 11}=\dfrac{(y^{2})^{2}}{\alpha ^{3}},$ $\alpha
_{\cdot 12}=\dfrac{-y^{1}y^{2}}{\alpha ^{3}},~\alpha _{\cdot 22}=\dfrac{%
(y^{1})^{2}}{\alpha ^{3}}$ and thus,%
\begin{equation*}
\kappa _{1}^{2}=\dfrac{F(\dot{y}^{2}y^{1}-\dot{y}^{1}y^{2})^{2}}{\alpha ^{3}}%
=F\alpha (\dfrac{d}{ds}(\arctan \dfrac{y^{2}}{y^{1}}))^{2}.
\end{equation*}%
Passing to (Euclidean)\ polar coordinates $y^{1}=\alpha \cos \theta ,$ $%
y^{2}=\alpha \sin \theta $ in the plane $(y^{1},y^{2}),$ we have $\theta
=\arctan \dfrac{y^{2}}{y^{1}}+const.$ and the first equation (\ref%
{2-dim_biharmonic})\ can be finally written as:%
\begin{equation}
\theta =\pm ~\kappa _{1}\int (F\alpha )^{-1/2}ds.  \label{theta}
\end{equation}

\bigskip

Consider now the second equation (\ref{2-dim_biharmonic}); the first term in
its right hand side is:%
\begin{equation}
\left\langle -R(D_{\partial _{s}}T),D_{\partial _{s}}T\right\rangle =-\kappa
_{1}^{2}\left\langle R(e_{2}),e_{2}\right\rangle =-\kappa _{1}^{2}F^{2}K=%
\dfrac{-3\kappa _{1}^{2}}{4F^{2}}\alpha ^{4}.  \label{Numata_term1}
\end{equation}

In order to calculate the second term $\left\langle P(D_{\partial
_{s}}T,D_{\partial _{s}}T),D_{\partial _{s}}T\right\rangle =P_{ijk}\tau
^{i}\tau ^{j}\tau ^{k},$ we make use of (\ref{local_expr_P}):%
\begin{equation*}
P_{ijk}=y^{l}C_{ijk|l}=y^{l}(C_{ijk,l}-G_{~l}^{h}C_{ijk\cdot h}-\Gamma
_{~il}^{h}C_{hjk}-\Gamma _{~jl}^{h}C_{ihk}-\Gamma _{~kl}^{h}C_{ijh}).
\end{equation*}

First of all, from:$\ C_{~ijk,l}=\dfrac{1}{4}(F_{~,l}^{2})_{\cdot ijk}=%
\dfrac{1}{2}(FF_{,l})_{\cdot ijk}=\dfrac{1}{4}(F\alpha _{\cdot
l}^{2})_{\cdot ijk},$ it follows, with the help of (\ref{Numata1}) and (\ref%
{Numata3}), that $y^{l}C_{ijk,l}=0.$

Then, taking into account that $y^{l}\Gamma _{~il}^{h}=G_{~\cdot i}^{h},$ $%
y^{l}G_{~l}^{h}=2G^{h},$ the expression (\ref{Numata_G}) of $G^{i}$ and of
the fact that $C_{ijk}$ is homogeneous of degree -1 in $y,$ we obtain:%
\begin{equation}
P_{ijk}=-\dfrac{\alpha ^{2}}{2F}C_{ijk}.  \label{P}
\end{equation}

Substituting (\ref{Numata2}) in (\ref{P}), contracting with $\tau ^{i}\tau
^{j}\tau ^{k}$ and then using (\ref{Numata4}), we find: $P_{ijk}\tau
^{i}\tau ^{j}\tau ^{k}=-\dfrac{\alpha ^{2}}{4}F_{\cdot ijk}\tau ^{i}\tau
^{j}\tau ^{k}=-\dfrac{\alpha ^{2}}{4}\alpha _{\cdot ijk}\tau ^{i}\tau
^{j}\tau ^{k}.$ Finally, by (\ref{Numata3}), (\ref{Numata5}) and (\ref%
{Numata_k}), we find: 
\begin{equation*}
\left\langle P(D_{\partial _{s}}T,D_{\partial _{s}}T),D_{\partial
_{s}}T\right\rangle =\dfrac{3\kappa _{1}^{2}}{4}(\dfrac{\alpha \alpha
^{\prime }}{F}+\dfrac{\alpha ^{4}}{F^{2}}).
\end{equation*}%
Using the latter and (\ref{Numata_term1}), together with $F=1$ in the second
equation (\ref{2-dim_biharmonic}), this becomes:%
\begin{equation*}
\dfrac{3\alpha \alpha ^{\prime }}{4}=-\kappa _{1}^{2},
\end{equation*}%
with the solution:%
\begin{equation*}
\alpha =(\mu s+\nu )^{1/2},
\end{equation*}%
where $\mu :=-\dfrac{8\kappa _{1}^{2}}{3}$ and $\nu \in \mathbb{R}$ are
constants. Once we have calculated $\alpha ,$ we are able to calculate $%
\theta $ (from (\ref{theta})), as:%
\begin{equation*}
\theta =\pm ~\kappa _{1}\int \alpha ^{-1/2}ds=\pm \dfrac{4\kappa _{1}}{3\mu }%
\left( \mu s+\nu \right) ^{3/4}+\gamma ,
\end{equation*}%
where $\gamma \in \mathbb{R}.$ Substituting into $y^{1}=\alpha \cos \theta ,$
$y^{2}=\alpha \sin \theta $ and integrating with respect to $s,$ we can
finally state:

\begin{proposition}
Proper biharmonic curves of $(M,F)$ are given, in the Finslerian natural
parametrization, by:
\end{proposition}

\begin{equation*}
\left\{ 
\begin{array}{c}
x^{1}=\int \alpha \cos \theta ds=\int (\mu s+\nu )^{1/2}\cos (\dfrac{4\kappa
_{1}}{3\mu }\left( \mu s+\nu \right) ^{3/4}+\gamma )ds \\ 
x^{2}=\int \alpha \sin \theta ds=\pm \int (\mu s+\nu )^{1/2}\sin (\dfrac{%
4\kappa _{1}}{3\mu }\left( \mu s+\nu \right) ^{3/4}+\gamma )ds.%
\end{array}%
\right. .
\end{equation*}

\subsubsection{A Randers-Minkowski norm on $\mathbb{R}^{3}$}

Let us find proper biharmonic curves of the Minkowski norm:%
\begin{equation}
F(y)=\sqrt{(y^{1})^{2}+(y^{2})^{2}+(y^{3})^{2}}+by^{3}=:\alpha +\beta
\end{equation}%
on $M=\mathbb{R}^{3}$. That is, we will look for non-constant $y=\dot{x}$
obeying the biharmonic equation and $F(y)=1$.

Let us denote: $\beta =b_{i}y^{i},$ $b_{i}:=b\delta _{3i}.$ The metric
tensor $g_{ij}$ is given, \cite{Shen}, by:%
\begin{equation}
g_{ij}=\dfrac{F}{\alpha }(\delta _{ij}-\alpha _{\cdot i}\alpha _{\cdot
j})+F_{\cdot i}F_{\cdot j}.  \label{metrics}
\end{equation}%
Also, there hold relations (\ref{Numata0})-(\ref{Numata3}) and, along any
unit speed curve $c$ on $(M,F)$, we have:%
\begin{eqnarray}
&&\tau ^{i}=\dfrac{dy^{i}}{ds},~\ \ F_{\cdot i}\tau ^{i}=F^{\prime
}=0,~\alpha _{\cdot i}\tau ^{i}=\alpha ^{\prime };  \label{RMink1} \\
&&\alpha _{\cdot ij}\tau ^{i}\tau ^{j}=F_{\cdot ij}\tau ^{i}\tau ^{j}=\dfrac{%
\kappa _{1}^{2}}{F},~\ \ F_{\cdot j}\dfrac{d\tau ^{j}}{ds}=-\dfrac{\kappa
_{1}^{2}}{F},  \label{RMink2} \\
&&\ \ \alpha _{\cdot ij}\tau ^{j}=\dfrac{1}{\alpha }(\delta _{ij}\tau
^{j}-\alpha _{i}\alpha ^{\prime }),~\ \alpha _{\cdot j}\dfrac{d\tau ^{j}}{ds}%
=\alpha ^{\prime \prime }-\dfrac{\kappa _{1}^{2}}{F};  \label{RMink3}
\end{eqnarray}%
We remark that here, $\kappa _{1}^{2}:=g_{ij}\tau ^{i}\tau ^{j}$ is not
necessarily constant.

\bigskip

Biharmonic curves are given by (\ref{first_integral}), i.e.:%
\begin{equation}
g_{ij}\dfrac{D\tau ^{j}}{ds}=\lambda _{i},~\ \lambda _{i}=const.
\label{first_integral1}
\end{equation}%
\ In the following, we will express these equations in terms of the
Euclidean metric $\delta _{ij}$ and of the derivatives of $\alpha .$

By (\ref{metrics}), (\ref{RMink2}), (\ref{RMink3}), we have:%
\begin{equation}
g_{ij}\dfrac{d\tau ^{j}}{ds}=\dfrac{F}{\alpha }\delta _{ij}\dfrac{d\tau ^{j}%
}{ds}+(\dfrac{\kappa _{1}^{2}}{\alpha }-\dfrac{F\alpha ^{\prime \prime }}{%
\alpha })\alpha _{\cdot i}-\dfrac{\kappa _{1}^{2}}{F}F_{\cdot i}.
\label{term1}
\end{equation}

For the second term in (\ref{first_integral1}), we get, with the help of (%
\ref{Numata1}), (\ref{Numata2}), (\ref{RMink1}), (\ref{RMink2}):%
\begin{equation*}
C_{ijk}\tau ^{j}\tau ^{k}=\dfrac{1}{2}(\dfrac{\kappa _{1}^{2}}{F}F_{\cdot
i}+F\alpha _{\cdot ijk}\tau ^{j}\tau ^{k})
\end{equation*}%
and further, using (\ref{Numata3}), (\ref{RMink3}):%
\begin{equation}
C_{ijk}\tau ^{j}\tau ^{k}=\dfrac{\kappa _{1}^{2}}{2F}F_{\cdot i}+(\dfrac{%
F(\alpha ^{\prime })^{2}}{\alpha ^{2}}-\dfrac{\kappa _{1}^{2}}{2\alpha }%
)\alpha _{\cdot i}-\dfrac{F\alpha ^{\prime }}{\alpha ^{2}}\delta _{ij}\tau
^{j}.  \label{term2}
\end{equation}

Adding (\ref{term1}) and (\ref{term2}), equations (\ref{first_integral})
become:%
\begin{equation}
\dfrac{F}{\alpha }\delta _{ij}\dfrac{d\tau ^{j}}{ds}-\dfrac{F\alpha ^{\prime
}}{\alpha ^{2}}\delta _{ij}\tau ^{j}+\left( \dfrac{F(\alpha ^{\prime })^{2}}{%
\alpha ^{2}}-\dfrac{F\alpha ^{\prime \prime }}{\alpha }+\dfrac{\kappa
_{1}^{2}}{2\alpha }-\dfrac{\kappa _{1}^{2}}{2F})\right) \alpha _{\cdot
i}=\lambda _{i}+\dfrac{\kappa _{1}^{2}}{2F}b_{i}.  \label{biharmonic_RM}
\end{equation}

Let us notice some more equalities, valid along any unit speed curve:%
\begin{eqnarray}
&&\alpha _{\cdot i}y^{i}=\alpha ,~\ \ \delta _{ij}y^{i}\tau ^{j}=\dfrac{1}{2}%
\dfrac{d\alpha ^{2}}{ds}=\alpha \alpha ^{\prime },~~\ \delta _{ij}\tau
^{i}\tau ^{j}=\dfrac{\alpha \kappa _{1}^{2}}{F}+(\alpha ^{\prime })^{2},\ 
\label{(*)} \\
&&\ \ \delta _{ij}y^{i}\dfrac{d\tau ^{j}}{ds}=\alpha \alpha ^{\prime \prime
}-\dfrac{\kappa _{1}^{2}\alpha }{F},~\delta _{ij}\tau ^{i}\dfrac{d\tau ^{j}}{%
ds}=\dfrac{\alpha ^{\prime }}{2F}\kappa _{1}^{2}+\dfrac{\alpha \kappa
_{1}\kappa _{1}^{\prime }}{F}+\alpha ^{\prime }\alpha ^{\prime \prime },
\label{(**)}
\end{eqnarray}%
(where, in the third relation (\ref{(*)}), we have used (\ref{metrics})).
Contracting these three equations (\ref{biharmonic_RM}) in turn with $%
y^{i},\tau ^{i}$ and $\delta ^{ik}\lambda _{k}$ and using (\ref{(*)}), (\ref%
{(**)}) together with $\alpha _{\cdot i}\tau ^{i}=\alpha ^{\prime }$, they
take the form:%
\begin{equation}
\left\{ 
\begin{array}{l}
\lambda _{i}y^{i}=-\kappa _{1}^{2} \\ 
\lambda _{i}\tau ^{i}=\kappa _{1}\kappa _{1}^{\prime } \\ 
\dfrac{F}{\alpha }\lambda _{j}\dfrac{d\tau ^{j}}{ds}-\dfrac{F\alpha ^{\prime
}}{\alpha ^{2}}\lambda _{j}\tau ^{j}+\left( \dfrac{F(\alpha ^{\prime })^{2}}{%
\alpha ^{2}}-\dfrac{F\alpha ^{\prime \prime }}{\alpha }+\dfrac{\kappa
_{1}^{2}}{2\alpha }-\dfrac{\kappa _{1}^{2}}{2F})\right) \delta ^{ij}\alpha
_{\cdot i}\lambda _{j}=-\kappa _{1}^{2}\gamma ,%
\end{array}%
\right.  \label{y}
\end{equation}%
where $\gamma :=-\kappa _{1}^{-2}\delta ^{ij}(\lambda _{i}+\dfrac{\kappa
_{1}^{2}}{2F}b_{i})\lambda _{j}$ is a constant.

Differentiating the first equation (\ref{y}) by $s,$ we have: $\lambda
_{i}\tau ^{i}=-2\kappa _{1}\kappa _{1}^{\prime },$ which, substituted into
the second one yields:%
\begin{equation*}
\kappa _{1}=const.
\end{equation*}%
(which was found in Section 5.1. to be true along any biharmonic curve).
Thus, we have: $\lambda _{i}\tau ^{i}=0\ \Rightarrow \lambda _{i}\dfrac{%
d\tau ^{i}}{ds}=0.$ Further, noticing that $\delta ^{ij}\alpha _{\cdot i}=%
\dfrac{y^{j}}{\alpha },$ the third equation (\ref{y}) becomes:%
\begin{equation}
\dfrac{1}{\alpha }(\dfrac{F(\alpha ^{\prime })^{2}}{\alpha ^{2}}-\dfrac{%
F\alpha ^{\prime \prime }}{\alpha }+\dfrac{\kappa _{1}^{2}}{2\alpha }-\dfrac{%
\kappa _{1}^{2}}{2F})=\gamma ;  \label{lambda}
\end{equation}%
taking into account that $F\equiv 1$ along $c,$ this is:%
\begin{equation}
(\ln \alpha )^{\prime \prime }-\dfrac{\kappa _{1}^{2}}{2\alpha }+\dfrac{%
\kappa _{1}^{2}}{2}+\gamma \alpha =0,  \label{alpha}
\end{equation}%
which obviously admits non-constant solutions.

Let $\alpha $ be a non-constant solution of (\ref{alpha}). Then,
substituting $\alpha =1-by^{3},$ we get a non-constant value for $y^{3}=\dot{%
x}^{3}.$ Further, $y^{1}=\dot{x}^{1}$ and $y^{2}=\dot{x}^{2}$ can be
obtained from the relations $F(y)=1,$ $\lambda _{i}y^{i}=-\kappa _{1}^{2}.$

\section{Conclusions}

In the present paper, we extend the notion of bienergy functional to curves
on Finsler spaces and deduce the equations of biharmonic curves as the
Euler-Lagrange equations attached to this functional; these equations turn
out to involve, besides the flag curvature (which is the analogue of the
curvature term in the Riemannian biharmonic equations), two specific,
non-Riemannian quantities, namely, the Cartan tensor and the Landsberg
tensor of the space. Still, just as in Riemannian spaces, all geodesics are
biharmonic curves, but the converse is generally not true.

A brief analysis in the Frenet frame shows that, though the equation of
biharmonic curves becomes more complicated compared to its Riemannian
counterpart, a basic result in Riemannian biharmonicity theory remains true
in our case, namely, any Finslerian biharmonic curve has constant geodesic
curvature.

Further, we focus on the conditions under which the implication \textit{%
biharmonic }$\rightarrow $ \textit{geodesic} becomes true. The first case to
be investigated is the one of smooth closed curves. In the particular case
of Riemannian spaces, this turns out to be the known condition of
nonpositive sectional curvature; still, in the general Finslerian case, this
condition is more complicated, involving the Cartan and Landsberg tensors.

A special attention is paid to two-dimensional Finsler spaces (Section 5.2).
We prove that, in dimension two, in: 1)\ Landsberg spaces with nonpositive
flag curvature and 2)\ flat (i.e., locally Minkowski)\ Finsler spaces, any
biharmonic curve is a geodesic.

Further, in Subsection 5.3., we investigate higher dimensional locally
Minkowski spaces. Here, we find a first integral of the biharmonic equation
and show that, unlike Euclidean spaces, flat Finsler spaces of dimension at
least three can admit proper unit speed biharmonic curves.

Two examples: a 2-dimensional projectively flat space and a 3-dimensional
locally Minkowski Finsler space, both admitting proper biharmonic curves,
are given in Section 5.4.

\textbf{Acknowledgment.} The work was supported by the Sectorial Operational
Program Human Resources Development (SOP HRD), financed from the European
Social Fund and by Romanian Government under the Project number
POSDRU/89/1.5/S/59323.

\end{document}